\documentclass[11pt]{article}

\usepackage{tikz}
\usepackage{subfigure}
\usepackage[english]{babel}

\usepackage[center]{caption2}
\usepackage{amsfonts,amssymb,amsmath,latexsym,amsthm}
\usepackage{multirow}
\usepackage[ruled,algosection,linesnumbered]{algorithm2e}
\usepackage[center]{caption2}
\usepackage{amsfonts,amssymb,amsmath,latexsym,amsthm}
\usepackage{multirow}
\usepackage[usenames,dvipsnames]{pstricks}
\usepackage{epsfig}
\usepackage{pst-grad} 
\usepackage{pst-plot} 
\usepackage[space]{grffile} 
\usepackage{etoolbox} 
\usepackage{indentfirst}
\makeatletter 
\patchcmd\Gread@eps{\@inputcheck#1 }{\@inputcheck"#1"\relax}{}{}

\newtheorem{thm}{Theorem}[section]
\newtheorem{cor}[thm]{Corollary}
\newtheorem{lem}[thm]{Lemma}
\newtheorem{prop}{Proposition}
\newtheorem{conj}[thm]{Conjecture}

\newtheorem{claim}{Claim}

\begin{document}

\title{Rainbow independent sets in graphs with maximum degree two
\thanks{The work was supported by NNSF of China (No. 12071453) and  Anhui Initiative in Quantum Information Technologies (AHY150200) and National Key Research and Development Project (SQ2020YFA070080).}
}
\author{Yue Ma$^a$, \quad Xinmin Hou$^{a,b}$,\quad Jun Gao$^a$,\quad Boyuan Liu$^a$, Zhi Yin$^a$\\
\small $^a$School of Mathematical Sciences\\
\small University of Science and Technology of China, Hefei, Anhui 230026, China.\\
\small $^{b}$ CAS Key Laboratory of Wu Wen-Tsun Mathematics\\
\small University of Science and Technology of China, Hefei, Anhui 230026, China.\\
\small xmhou@ustc.edu.cn
}

\date{}

\maketitle

\begin{abstract}
Given a graph $G$, let $f_{G}(n,m)$ be the minimal number $k$ such that every $k$ independent $n$-sets in $G$ have a rainbow $m$-set. Let $\mathcal{D}(2)$ be the family of all graphs with maximum degree at most two. Aharoni et al. (2019) conjectured that (i) $f_G(n,n-1)=n-1$ for all graphs $G\in\mathcal{D}(2)$ and (ii) $f_{C_t}(n,n)=n$ for $t\ge 2n+1$. Lv and Lu (2020) showed that the conjecture (ii) holds when $t=2n+1$. In this article, we show that the conjecture (ii) holds for $t\ge\frac{1}{3}n^2+\frac{44}{9}n$. Let $C_t$ be a cycle of length $t$ with  vertices being arranged in a clockwise order. An ordered set $I=(a_1,a_2,\ldots,a_n)$ on $C_t$ is called  a $2$-jump independent $n$-set of $C_t$ if $a_{i+1}-a_i=2\pmod{t}$ for any $1\le i\le n-1$. We also show that a collection of 2-jump independent $n$-sets $\mathcal{F}$ of $C_t$ with $|\mathcal{F}|=n$ admits a rainbow independent $n$-set, i.e. (ii) holds if we restrict $\mathcal{F}$ on the family of 2-jump independent $n$-sets. Moreover, we prove that if the conjecture (ii) holds, then  (i) holds for all graphs $G\in\mathcal{D}(2)$ with $c_e(G)\le 4$, where $c_e(G)$ is the number of components of $G$ isomorphic to cycles of even lengths.
\end{abstract}

\section{Introduction}
Let $\mathcal{F}=(F_1,F_2,...,F_n)$ be a collection of sets (not necessarily distinct), a {\it (partial) rainbow set} of $\mathcal{F}$ is a set $R=\{x_{i_1},x_{i_2},...,x_{i_m}\}$ such that  $x_{i_j}\in F_{i_j}$ for  $1\le i_1< i_2<\ldots< i_k\le m\le n$. Given a graph $G$
and a collection $\mathcal{F}$ of independent sets of $G$, a rainbow set $R$ of $\mathcal{F}$ is called a {\it rainbow independent set} of $(\mathcal{F}, G)$ if $R$ is also an independent set of $G$.
An {\it $m$-set} is a set of size $m$. Let $\mathcal{I}$ be a family of sets. Write $\mathcal{F}\sqsubseteq \mathcal{I}$ for a collection of sets  (not necessarily distinct) $\mathcal{F}$ with each member of $\mathcal{F}$ belonging to $\mathcal{I}$.
Given a graph $G$ and an integer $n$, write $\mathcal{I}(G)$ (resp. $\mathcal{I}_n(G)$ or $\mathcal{I}_{n^+}(G)$) for the family of independent sets (resp. independent sets such that each of them has uniform size $n$ or of size at least $n$) of $G$.
Given a graph $G$ and integers $m,n$ with $m\le n$, define
$$f_G(n,m)=\min\{|\mathcal{F}| : \mathcal{F}\sqsubseteq \mathcal{I}_n(G) \mbox{ and $(\mathcal{F}, G)$ has a rainbow indepent $m$-set}\}.$$
Given a family $\mathcal{G}$ of graphs, let
$$f_{\mathcal{G}}(n,m)=\sup\{f_G(n,m) : G\in\mathcal{G}\}.$$

Clearly, a rainbow matching in a graph is a rainbow independent set in its line graph. Partially motivated by the study of the rainbow matching problem in graphs (there are fruitful results about the problem, one can refer~\cite{AB09,ABCHS19,ABKK21,AHJ19,AKZ18,BGS17,D98} for more details), Aharoni, Briggs, Kim and Kim studied the rainbow independent set problem in certain classes of graphs and proposed several conjectures in~\cite{ABKK19}.

Let $m\le n$. Clearly,
\begin{eqnarray}
f_{G}(n,m)  &\ge & m. \label{EA: e1}
\end{eqnarray}
Aharoni, et al.~\cite{ABKK19} conjectured that the lower bound in~(\ref{EA: e1}) is tight for $f_G(n,n-1)$ with $\Delta(G)\le 2$ and $f_{C_t}(n,n)$, where $C_t$ is a cycle with $t$ vertices.
Write $\mathcal{D}(2)$ for the family of graphs with maximum degree two.

\begin{conj}[Conjecture 2.14 in~\cite{ABKK19}]\label{CONJ: 1}
 $f_{\mathcal{D}(2)}(n, n-1)=n-1$.
\end{conj}

\begin{conj}[Conjecture 2.9 in~\cite{ABKK19}]\label{conj}
 If $t\ge 2n+1$, then $f_{C_t}(n,n)=n$.
\end{conj}

The following proposition can be easily checked  (also has been observed in~\cite{ABCHS19,ABKK19,D98}).
\begin{prop}\label{PROP: p1}
$f_{C_{2n}}(n,n)=2n-1$.
\end{prop}

When $t=2n+1$, Conjecture~\ref{conj} was confirmed by Lv and Lu~\cite{Lu20}.
\begin{thm}
[Theorem 1 in~\cite{Lu20}]\label{THM: Lu}
$f_{C_{2n+1}}(n,n)=n$.
\end{thm}

In this article, we first show that Conjecture~\ref{CONJ: 1} is true when $G$ is 2-regular and $|V(G)|\in\{ 2n-1, 2n\}$ and then we show that Conjecture~\ref{conj} seems stronger than Conjecture~\ref{CONJ: 1}, i.e. Conjecture~\ref{conj} implies Conjecture~\ref{CONJ: 1} for graphs $G\in\mathcal{D}(2)$ with $c_e(G)\le 4$, where $c_e(G)$ is the number of components of $G$ isomorphic to cycles of even lengths. So we concentrate our attention on Conjecture~\ref{conj} and prove that this  conjecture holds when $t$ is large. The main results of the article are listed below.

\begin{thm}\label{THM: Con1}
If $G$ is 2-regular with $2n-1\le |V(G)|\le 2n$ then $f_G(n, n-1)=n-1$.
\end{thm}

\begin{thm}\label{THM: EQ}
If $f_{C_\ell}(n,n)=n$ for $\ell\ge 2n+1$, then  $f_G(n, n-1)=n-1$ for all graphs $G\in\mathcal{D}(2)$ with $c_e(G)\le 4$, provided that $\mathcal{I}_n(G)\not=\emptyset$.
\end{thm}

\begin{thm}\label{large}
 $f_{C_t}(n,n)=n$ for $t>\frac{1}{3}n^2+\frac{44}{9}n$.
\end{thm}

Given integers $a,b$ with $a\le b$, let $[a,b]=\{a,a+1,\ldots,b-1,b\}$, and write $[b]$ for $[1,b]$ for simplicity.
Let $C_t$ be a cycle with vertex set $[a,a+t-1]$ and edge set $\{a(a+1),\ldots,(a+t-2)(a+t-1),(a+t-1)a\}$. For $t\ge 2n+1$ and $2\le k\le t-2$, an  ordered set $I=(a_1,a_2,\ldots,a_n)$ is called  a {\it $k$-jump independent set} of $C_t$ if $a_{i+1}-a_i=k\pmod{t}$ for any $1\le i\le n-1$.
We call $a_1$ the {\it start} of $I$, denoted by $s(I)$, and $a_n$ the {\it end} of $I$.
Let $\mathcal{I}^k_n(G)$ be the family of all $k$-jump independent $n$-sets in $G$. We prove that  Conjecture~\ref{conj} is true if we restrict  $\mathcal{F}\sqsubseteq \mathcal{I}^2_n(C_t)$.

\begin{thm}\label{+2}
Given integers $t,n$ with $t\ge 2n+1$, if $\mathcal{F}\subseteq \mathcal{I}^2_n(C_t)$ with $|\mathcal{F}|=n$, then $\mathcal{F}$ has a rainbow independent $n$-set.
\end{thm}

\noindent{\bf Remark:} Any independent $n$-set of $C_{2n+1}$  must be a 2-jump set. So the theorem can be viewed as a generalization of Theorem~\ref{THM: Lu} in some sense.

We give a bit more definitions and notation. For a collection $\mathcal{F}$ of sets and a given set $A$, denote $\mathcal{F}-A=\{I-A : I\in \mathcal{F}\}$ and $\mathcal{F}\cap A=\{I\cap A : I\in\mathcal{F}\}$. For a set $B$ of numbers and a given number $i$, let $B+i=\{b+i : b\in B\}$.
For a graph $G$ and a collection $\mathcal{F}$ of independent sets in $G$, let $v\in V(G)$, define $C_{\mathcal{F}}(v)=\{I : v\in I\in\mathcal{F}\}$ to be the {\it list} of $v$ and $c_{\mathcal{F}}(v)=|C_{\mathcal{F}}(v)|$ the {\it list number} of $v$. For a rainbow independent set $R$ of $\mathcal{F}$ and $v\in R$, let $C_{R}(v)$ be the color (i.e. the independent set in the list $C_{\mathcal{F}}(v)$) assigned to $v$ and let $C_{R}=\cup_{v\in R}\{C_{R}(v)\}$.

The rest of the paper is arranged as follows. We  give a greedy algorithm to find a rainbow independent set for a given graph $G$ and $\mathcal{F}\sqsubseteq \mathcal{I}(G)$ and some preliminaries in Section 2. We prove Theorems~\ref{THM: Con1} and \ref{THM: EQ} in Section 3. In Sections 4 and 5, we give the proofs of Theorem~\ref{large} and \ref{+2}. We also give some discussions in the last section.

\section{A greedy algorithm and some preliminaries}

First, we give a greedy algorithm (GRIS) to find a rainbow independent set in a given graph $G$ and $\mathcal{F}\sqsubseteq \mathcal{I}(G)$.
\begin{algorithm}\label{ALG: GR}
\small
\SetAlgoLined
\KwIn{A graph $G$ with ordered vertex set $A=\{a_1,a_2,\ldots,a_t\}$, and a collection of independent set $\mathcal{F}=(I_1,...,I_k)\sqsubseteq \mathcal{I}(G)$}
\KwOut{A rainbow independent set $R$ of $\mathcal{F}$}

 Set $R=C=\emptyset$ and $j=0$.

 Reset $j:=j+1$. If $C_{\mathcal{F}}(a_j)\setminus C=\emptyset$ or $C_{\mathcal{F}}(a_j)\setminus C\neq\emptyset$ but $R\cup\{a_j\}$ is not an independent set, go to step 3; otherwise, reset $R:=R\cup\{a_j\}$  and choose $I_i\in C_{\mathcal{F}}(a_j)\backslash C$ with the minimal index $i$ and reset $C:= C\cup\{I_i\}$.

If $j<t$, return to step 2; else if $j=t$, then stop and output $R=\mbox{GRIS}(G, \mathcal{F})$ and $C=C(G,\mathcal{F})$.
We call $C(G,\mathcal{F})$ a {\it greedy color set}.

\caption{GREEDY-RAINBOW-INDEPENDENT-SET (GRIS($G, \mathcal{F})$)}
\end{algorithm}

It is easy to check that the output $R=\mbox{GRIS}(G, \mathcal{F})$  is a rainbow independent set of $(\mathcal{F}, G)$. The following is a simple fact when we apply GRIS to a path.

\begin{lem}\label{path}
Let $P_t$ be a path with vertex set $[t]$ and edge set $\{12, 23, \ldots, (t-1)t\}$ and  $\mathcal{F}\sqsubseteq \mathcal{I}_{(n-1)^+}(P_t)$ with $|\mathcal{F}|\ge n-1$. Suppose $t\ge 2n-1$.  Let $R=\mbox{GRIS}(P_t, \mathcal{F})$ and $C=C(P_t, \mathcal{F})$. If $|R|<n$, then $|R|=n-1$ and, for any $I\in\mathcal{F}\backslash C$, we have $|I|=n-1$ and $|I\cap\{a,a+1\}|=1$ for any $a\in R$.
\end{lem}
\begin{proof}
Suppose $|R|=|C|=k<n$. Pick $I\in\mathcal{F}\setminus C$. Then, for any $i\in I$, $R\cup\{i\}$ can not be a larger independent set  of  $P_t$. Thus $i\in R$ or $i$ is a neighbor of some vertex in $R$, i.e. $i\in(R-1)\cup R\cup( R+1)$.
If $i\notin R\cup(R+1)$, then $i\in R-1$, i.e., $i+1\in R$. But at step $i$ of the algorithm GRIS, $i$ will be added to $R$ since $i-1\notin R$. This is a contradiction to $i+1\in R$. Hence $i\in R\cup(R+1)$.
Since $I$ is an independent set,  $|I\cap\{a,a+1\}|\le 1$ for any $a\in R$. So $n-1\le |I|\le|R|=k<n$. Thus we have $|I|=|R|=k=n-1$ and  $|I\cap\{a,a+1\}|=1$ for any $a\in R$.
\end{proof}

In Lemma~\ref{path}, if we choose $\mathcal{F}\in\mathcal{I}_n(P_t)$ with $|\mathcal{F}|=n$ then $|R|$ must be $n$. So we have the following corollary.
\begin{cor}\label{COR: f_P(n,n)}
Suppose $t\ge 2n-1$. Then $f_{P_t}(n,n)=n$.
\end{cor}

\noindent{\bf Remark:} Note that Aharoni et al. have proved that for the family of  chordal graphs $\mathcal{T}$ and $m\le n$, $f_{\mathcal{T}}(n,m)=m$ (Theorem 3.20 in~\cite{ABKK19}). So, in fact,  we have the following result.
\begin{cor}\label{COR: f_P(n,m)}
Suppose $t\ge 2n-1$ and $F_t$ is a forest of order $t$. Then $f_{F_t}(n,m)=m$ for $m\le n$.
\end{cor}

\begin{cor}\label{no0}
Let $C_t$ be a cycle with vertex set $[t]$ and edge set $\{12, 23, \ldots, (t-1)t, t1\}$. Suppose $t\ge 2n$. Then we have

(A)  $f_{C_t}(n,n-1)=n-1$.

(B) Let $m$ be the maximum integer such that  $f_{C_t}(n,m)=n$. Then $m\ge n-1$. Suppose $\mathcal{F}\sqsubseteq \mathcal{I}_{n}(C_t)$ with $|\mathcal{F}|=n$. If $m=n-1$ then the following hold.

(B1) $c_{\mathcal{F}}(i)>0$ for any $i\in [t]$.

(B2) If $C_{\mathcal{F}}(i)=I$ for some $i\in[t]$, then $I\in C_{\mathcal{F}}(i\pm 2)$.

(B3) If $c_{\mathcal{F}}(i)=1$ for some $i\in[t]$ and $t\ge 2n+1$, then $c_{\mathcal{F}}(i\pm 1)>1$.

\end{cor}
\begin{proof}
Suppose $\mathcal{F}\sqsubseteq\mathcal{I}_n(C_t)$ with $|\mathcal{F}|=k$. Let $P_{t-1}=C_t-\{t\}$ and $\mathcal{F}'=\mathcal{F}-\{t\}$.
Then $\mathcal{F}'\sqsubseteq\mathcal{I}_{(n-1)^+}(P_{t-1})$. Note that $t-1\ge2n-1$. We can apply the algorithm GRIS to $P_{t-1}$ and $\mathcal{F}'$.
Let $R=\mbox{GRIS}(P_{t-1}, \mathcal{F}')$. Then $|R|\le |\mathcal{F}'|=k$.

(A) If $k=n-1$ then $|R|\le k<n$. By Lemma~\ref{path}, $|R|=n-1$. Clearly, $R$ is a rainbow independent set of $(\mathcal{F},C_t)$ too. This implies that $f_{C_t}(n,n-1)=n-1$.

(B) Suppose $k=n$.  Then $|R|\le k=n$. If $|R|<n$, by Lemma~\ref{path}, we have $|R|=n-1$.
So we have $m\ge |R|\ge n-1$. Now we assume $m=n-1$.

(B1) If there is a vertex $i\in[t]$ with $c_{\mathcal{F}}(i)=0$, without loss of generality, assume $i=t$, then $\mathcal{F}'=\mathcal{F}-\{t\}=\mathcal{F}\sqsubseteq\mathcal{I}_n(P_{t-1})$. Note that $t-1\ge 2n-1$. By Corollary~\ref{COR: f_P(n,n)}, we can find a rainbow independent $n$-set $R'$ of $(\mathcal{F}', P_{t-1})$, which is also a rainbow independent $n$-set of $(\mathcal{F}, C_t)$, a contradiction to the maximality of $m$.

(B2) Without loss of generality, assume $C_{\mathcal{F}}(t)=\{I\}$. By the symmetry of $C_t$, it is sufficient to prove $2\in I$.
Recall that $\mathcal{F}'=\mathcal{F}-\{t\}\sqsubseteq\mathcal{I}_{(n-1)^+}(P_{t-1})$. Let $C=C(P_{t-1}, \mathcal{F}')$. Note that $m=|R|=|\mbox{GRIS}(P_{t-1}, \mathcal{F}')|=|C|=n-1$ and $|\mathcal{F}'|=n$. We have $|\mathcal{F}'\setminus C|=1$. Let $\mathcal{F}'\setminus C=\{I'\}$. By Lemma~\ref{path}, we have $|I'|=n-1$ and $|I'\cap\{a,a+1\}|=1$ for every $a\in R$.
Since $I$ is the only independent set containing $t$ in $\mathcal{F}$, we have $I'=I\setminus \{t\}$. By (B1), $C_{\mathcal{F}'}(1)=C_{\mathcal{F}}(1)\neq\emptyset$. So by the processing of the  algorithm GRIS, $1$ is the first vertex added to $R$. Hence $|(I\setminus\{t\})\cap\{1,2\}|=1$. Since $1t\in E(C_t)$ and $t\in I$, we have $1\notin I$. This implies that $2\in I$.

(B3) If not, there are two consecutive vertices $i, i+1\in[t]$ with $|C_{\mathcal{F}}(i)|=|C_{\mathcal{F}}(i+1)|=1$. Without loss of generality, assume  $i=t-3$ and $C_{\mathcal{F}}(t-3)=\{I_1\}$ and $C_{\mathcal{F}}(t-2)=\{I_2\}$.
Let $P_{t-6}=C_t-[t-5,t]$ and $\mathcal{F}'=(\mathcal{F}\setminus\{I_1,I_2\})-[t-5,t]$.
Since $I_1$  is the only member of $\mathcal{F}$ containing $t-3$  and $I_2$  is the only one of $\mathcal{F}$ containing $t-2$, we have $I\cap\{t-3, t-2\}=\emptyset$ for any $I\in\mathcal{F}\setminus\{I_1, I_2\}$. So $|I\cap[t-5,t]|\le 2$ for any $I\in\mathcal{F}\setminus\{I_1, I_2\}$.
Thus $\mathcal{F}'\sqsubseteq\mathcal{I}_{(n-2)^+}(P_{t-6})$.
Note that $t-6\ge 2n+1-6=2(n-2)-1$. By Corollary~\ref{COR: f_P(n,n)}, we have $f_{P_{t-6}}(n-2,n-2)=n-2$.
So we have a rainbow independent $(n-2)$-set $R$ of $(P_{t-6}, \mathcal{F}')$.  This is also a rainbow independent $(n-2)$-set of $(\mathcal{F}\setminus\{I_1,I_2\}, C_t)$.
Recall that $C_{\mathcal{F}}(t-3)=\{I_1\}$ and $C_{\mathcal{F}}(t-2)=\{I_2\}$. By (B2), $I_1\in C_{\mathcal{F}}(t-1)$ and $I_2\in C_{\mathcal{F}}(t-4)$.
Therefore, by adding $t-4$ and $t-1$ to $R$, we get a rainbow independent $n$-set $R\cup\{t-1,t-4\}$ of $( \mathcal{F}, C_t)$, a contradiction to $m=n-1$.

\end{proof}


\section{Proofs of Theorem~\ref{THM: Con1} and \ref{THM: EQ}}
An {\it odd (resp. even) cycle} is a cycle of odd length (resp. even length).
\begin{proof}[Proof of Theorem~\ref{THM: Con1}]
We prove by induction on the number of components  $c(G)$ of $G$.
If $c(G)=1$, then $G$ is a cycle. By Corollaries~\ref{no0} (A), we have $f_G(n,n-1)=n-1$.
Now suppose $c(G)>1$ and the result holds for all 2-regular graphs $H$ with $c(H)<c(G)$.
Let $\mathcal{F}=(I_1, I_2, \ldots, I_{n-1})\sqsubseteq\mathcal{I}_n(G)$.

If $G$ has a component isomorphic to a cycle $C_{2m+1}$ for some $m\in [n-1]$, then $|I_i\cap V(C_{2m+1})|=m$ for every $i\in[n-1]$ because $2n-1\le|V(G)|\le 2n$.
Let $\mathcal{F}_1=(I_1\cap V(C_{2m+1}), \ldots, I_m\cap V(C_{2m+1}))$. Then $\mathcal{F}_1\sqsubseteq\mathcal{I}_m(C_{2m+1})$.
By Theorem~\ref{THM: Lu}, we can find a rainbow independent $m$-set $R_1$ of $(\mathcal{F}_1, C_{2m+1})$.
Let $H=G-V(C_{2m+1})$ and $\mathcal{F}_2=(I_{m+1}\cap V(H), \ldots, I_{n-1}\cap V(H))$. Then $H$ is 2-regular with $|V(H)|=2(n-m)-1$ and $\mathcal{F}_2\sqsubseteq\mathcal{I}_{n-m}(H)$.
By the induction hypothesis, there is a rainbow independent $(n-m-1)$-set $R_2$ of $(\mathcal{F}_2, H)$. So $R_1\cup R_2$ is a rainbow independent $(n-1)$-set of $(\mathcal{F}, G)$, we are done.


Now suppose that every component of $G$ is an even  cycle. Let $G=C_{1}\cup C_{2}\cup\ldots\cup C_{k}$ with $|V(C_i)|=2n_i$ for $i\in[k]$.
Then $|I_i\cap V(C_{j})|=n_j$ for every $i\in[n-1]$ and $j\in[k]$.
Let $H=C_{2}\cup\ldots \cup C_{k}$. Then $\mathcal{F}'=(I_{n_1+1}\cap V(H), \ldots, I_{n-1}\cap V(H))\sqsubseteq\mathcal{I}_{n-n_1}(H)$.
By the induction hypothesis, there is a rainbow independent $(n-n_1-1)$-set $R'$ of $(\mathcal{F}', H)$. So $|R'\cap V(C_{j})|=n_j$ for all $j\in[2,k]$ but exactly one exception, without loss of generality, say $C_{2}$, with $|R'\cap V(C_{2})|=n_2-1$. Let $J=\cup_{v\in R'\cap V(C_{2})} C_{R'}(v)$. Then $|J|=n_2-1$.
Let $\mathcal{F}_1=(I_j : j\in [n_1] \text{ or } I_j\in J)$ and assume $n_1\le n_2$. Then $|\mathcal{F}_1|=n_1+n_2-1\ge 2n_1-1$ and $\mathcal{F}_1\cap V(C_{1})\sqsubseteq\mathcal{I}_{n_1}(C_{1})$. So we have a rainbow independent $n_1$-set $R_1$ of $(\mathcal{F}_1\cap V(C_1), C_{1})$ by Proposition~\ref{PROP: p1}. Let $\mathcal{F}_2=\mathcal{F}_1\setminus C_{R_1}$. Then $|\mathcal{F}_2|=n_2-1$ and $\mathcal{F}_2\cap V(C_{2})\sqsubseteq\mathcal{I}_{n_2}(C_{2})$. By Corollary~\ref{no0} (A),  there is a  rainbow independent $(n_2-1)$-set of $(\mathcal{F}_2\cap V(C_{2}), C_{2})$. Therefore, $R'\setminus(R'\cap V(C_{2}))\cup R_1\cup R_2$ is rainbow independent $(n-1)$-set of $(\mathcal{F}, G)$.  This completes the proof.
\end{proof}

Let $V$ be a finite set and $\mathcal{I}\sqsubseteq 2^V$. For a subset $S\subset V$, define $$h(\mathcal{I},S)=\max\{m : |\{I\in\mathcal{I} : |I\cap S|\ge m\}|\ge m\}.$$

\begin{lem}\label{LEM: 2cycles}
Let $W$ be a finite set and $V\subseteq W$. Suppose $\mathcal{I}\sqsubseteq 2^W$ is a collection of $n^+$-sets with $|\mathcal{I}|=n$ and $h(\mathcal{I},V)=n$. Let $(V_1,V_2)$ be a partition of $V$ and let $h(\mathcal{I},V_i)=m_i$ for $i=1,2$. Then we have $m_1+m_2\geq n$.  Furthermore,

(i) for each $\ell\leq m_1+m_2-n$, there exists a partition $(\mathcal{I}_1,\mathcal{I}_2)$ of $\mathcal{I}$ such that $h(\mathcal{I}_1,V_1)=m_1-\ell$ and $h(\mathcal{I}_2,V_2)=n-m_1+\ell$;

(ii) if $m_1+m_2-n=0$, then $|I\cap V_{3-i}|\leq m_{3-i}$ for every $I\in \mathcal{I}_i$ and $i=1,2$;

(iii) if $m_1+m_2-n>0$ and $\ell\leq m_1+m_2-n-1$, then we can choose $\mathcal{I}_2$ with $|I\cap V_2|\geq n-m_1+\ell+1$ for every $I\in \mathcal{I}_2$.
\end{lem}

\begin{proof}
Assume $\mathcal{I}=(I_1, I_2, \ldots, I_n)$ with $|I_1\cap V_1|\geq|I_2\cap V_1|\geq\ldots\geq|I_n\cap V_1|$. By the definition of $m_1$, $|I_i\cap V_1|\geq m_1$ for each $1\leq i\leq m_1$ and $|I_j\cap V_1|\leq m_1$ for each $m_1+1\leq j\leq n$. So we have  $|I_j\cap V_2|\ge n-m_1$ for each $m_1+1\leq j\leq n$.  By the definition of $m_2$,  we have $m_2\geq n-m_1$, i.e. $m_1+m_2\ge n$.

(i) Let  $\ell\leq m_1+m_2-n$ and $\mathcal{I}_1'=\{I_i: |I_i\cap V_1|\geq m_1-\ell\}$ and $\mathcal{I}_2'=\{I_i: |I_i\cap V_2|\geq n-m_1+\ell\}$. By the definition of $m_1$ and $m_2$, we have $|\mathcal{I}_1'|\geq m_1$ and $|\mathcal{I}_2'|\geq m_2$. Since every $|I_i|\geq n$, we have either $|I_i\cap V_1|\geq m_1-\ell$ or $|I_i\cap V_2|> n-m_1+\ell$. Hence $\mathcal{I}_1'\cup\mathcal{I}_2'=\mathcal{I}$ and so $|\mathcal{I}_1'\setminus\mathcal{I}_2'|\leq n-m_2\leq m_1-\ell$. Therefore, we can choose a subset $\mathcal{I}_1$ of $\mathcal{I}$ with $\mathcal{I}_1'\setminus\mathcal{I}_2'\subseteq\mathcal{I}_1\subseteq\mathcal{I}_1'$ and $|\mathcal{I}_1|=m_1-\ell$. Let $\mathcal{I}_2=\mathcal{I}\setminus\mathcal{I}_1$. Clearly, $h(\mathcal{I}_1,V_1)=m_1-\ell$ and $h(\mathcal{I}_2,V_2)=n-m_1+\ell$.

(ii) Clearly, if $\ell=m_1+m_2-n=0$ then $|I\cap V_{3-i}|\le n-m_i=m_{3-i}$ for any $I\in\mathcal{I}_i$ and $i=1,2$.

(iii) If $m_1+m_2-n>0$ and $\ell\le m_1+m_2-n-1$, then we can reset $\mathcal{I}_2'=\{I_i: |I_i\cap V_2|\geq n-m_1+\ell+1\}$. Note that $n-m_1+\ell+1\leq m_2$. All discussions in (i) keep true. Thus, we have the desired partition $(\mathcal{I}_1, \mathcal{I}_2)$ of $\mathcal{I}$ with $|I\cap V_2|\geq n-m_1+\ell+1$ for every $I\in \mathcal{I}_2$.
\end{proof}

Now we give the proof of Theorem~\ref{THM: EQ}.
\begin{proof}[Proof of Theorem~\ref{THM: EQ}:]
Suppose to the contrary that there is a graph $G\in\mathcal{D}(2)$ with $c_e(G)\le 4$ and an $\mathcal{F}=(I_1, I_2, \ldots, I_{n-1})\sqsubseteq\mathcal{I}_{n}(G)$ such that $G$ contains no rainbow independent $(n-1)$-set of $(\mathcal{F}, G)$.
Since $G\in\mathcal{D}(2)$, each component of $G$ is a path or a cycle. Assume that $G$ is a minimum counterexample.
We claim that $G$ contains no component isomorphic to a path or an odd cycle.
\begin{claim}\label{CLM: c2}
$G$ contains no component $H$ with $f_H(h, h)=h$, where $h=h(\mathcal{F}, V(H))$.  In particular, $G$ contains no component isomorphic to a path or an odd cycle.
\end{claim}
\begin{proof}
Suppose to the contrary that there is a component $H$ of $G$ with $f_{H}(h,h)=h$, where $h=f(\mathcal{F}, V(H))$.  Let $G'=G-V(H)$.  Without loss of generality, assume $|I_1\cap V(H)|\geq|I_2\cap V(H)|\geq\ldots\geq|I_{n-1}\cap V(H)|$.  So $|I_i\cap V(H)|\ge h$ for every $1\le i\le h$ and $|I_j\cap V(G')|\le h$ for every $h+1\le j\le n-1$ by the definition of $h$. Let $\mathcal{F}_1=(I_1, \ldots, I_h)$ and $\mathcal{F}_2=(I_{h+1},\ldots, I_{n-1})$. Then $\mathcal{F}_1\sqsubseteq\mathcal{I}_{h^+}(H)$ and $\mathcal{F}_2\sqsubseteq\mathcal{I}_{(n-h)^+}(G')$.
Since $f_{H}(h,h)=h$, we can find a rainbow independent $h$-set $R_1$ of  $(\mathcal{F}_1, H)$.
Since $G$ is a minimum counterexample, we can find a rainbow independent $(n-h-1)$-set $R_2$ of $(\mathcal{F}_2, G')$.
Therefore, $R_1\cup R_2$ is a rainbow independent $(n-1)$-set of $(\mathcal{F}, G)$, a contradiction.

If $H$ is isomorphic to a path or an odd cycle. Then $|V(H)|\ge 2h+1$  if $H$ is an odd cycle and $2h$ otherwise. By the assumption $f_{C_s}(h,h)=h$ when $s\ge 2h+1$ and Corollary~\ref{COR: f_P(n,m)}, we always have $f_{H}(h,h)=h$. We are done.

\end{proof}

By Claim~\ref{CLM: c2}, we may assume that $G$ consists of $k$ even cycles, namely $G=\cup_{i=1}^{k} C_i$, where $|V(C_i)|=2n_i, 1\leq i\leq k$.  By Corollary~\ref{no0} (A),  $k\ge 2$. Let $V_i=V(C_i)$.  By Claim~\ref{CLM: c2}, $h(\mathcal{F}, V_i)=n_i$ for $i\in[k]$. Let $\mathcal{F}_i=\{I_j\in \mathcal{F} : |I_j\cap V_i|=n_i\}$, $i\in[k]$. Then $|\mathcal{F}_i|\geq n_i$ for $i\in[k]$. Denote $t= \sum_{i=1}^{k} n_i-n$.

\begin{claim}\label{CLM: t>k-1}
 $t\geq k-1$.
\end{claim}
\begin{proof}
Suppose to the contrary that $t\leq k-2$.
Note that $n=|I|=\sum_{i=1}^k|I\cap V_i|$ for every $I\in\mathcal{F}$. So there are at least two cycles $C_i, i\in [k]$ with $|I\cap V_i|\ge n_i$, i.e., every $I$ lies in at least two of $\{\mathcal{F}_i, i\in[k]\}$ for every $I\in \mathcal{F}$.
Therefore, $$\sum_{i=1}^{k}|\mathcal{F}_i|\geq 2|\mathcal{F}|=2(n-1)=2\sum_{i=1}^{k}n_i-2t-2>\sum_{i=1}^{k}(2n_i-2).$$
 Thus there exists at least one $i\in[k]$, say $i=1$, such that $|\mathcal{F}_1|\geq 2n_1-1$. By Proposition~\ref{PROP: p1}, we can find a rainbow independent $n_1$-set $R_1$ of $(\mathcal{F}_1, C_1)$.  By the minimality of $G$, we can find a rainbow independent $(n-n_1-1)$-set $R_2$ of $(\mathcal{F}\setminus C_{R_1}, G-V_1)$. Hence $R_1\cup R_2$ is a rainbow independent $(n-1)$-set of $(\mathcal{F}, G)$, a contradiction.
\end{proof}

Let $n_k=\min\{n_1, \ldots, n_k\}$. By Corollary~\ref{no0} (A), there is a rainbow independent $(n_k-1)$-set $R_k$ of $(\mathcal{F}_k, C_k)$. Let $\mathcal{A}=\mathcal{F}\setminus C_{R_k}$.
Then $$|\mathcal{A}|=n-1-(n_k-1)=\sum_{i=1}^{k-1}n_i-t$$ and $$h(\mathcal{A}, G-V_k)=\sum_{i=1}^{k-1}n_i-t$$ because $|I\cap V(G-V_k)|\ge n-n_k=\sum_{i=1}^{k-1}n_i-t$ for every $I\in\mathcal{A}$.
If $k=2$, then $G-V_k=C_1$ and so $h(\mathcal{A}, C_1)=n_1-t<n_1$ because $t\geq k-1=1$. By the assumption $f_{C_s}(n,n)=n$ for $s\ge 2n+1$, we have a rainbow independent $(n_1-t)$-set $R_1$ of $(\mathcal{A}, C_1)$.  Hence $R_1\cup R_2$ is a rainbow independent $(n-1)$-set of $(\mathcal{F}, G)$, a contradiction.

Now assume that $k\geq 3$.  Let $\hat{V_i}=V(G-V_k-V_i)$ and $h(\mathcal{A}, V_i)=m_i$ for $i\in [k-1]$. Then $\hat{V_i}\neq \emptyset$. Applying Lemma~\ref{LEM: 2cycles} to $V_i\cup\hat{V_i}$, we have $m_i+h(\mathcal{A}, \hat{V_i})\ge |\mathcal{A}|$.

\begin{claim}\label{2part}
For $i\in[k-1]$, $$m_i+h(\mathcal{A}, \hat{V_i})=\left\{\begin{array}{ll}
                                                         |\mathcal{A}|+1 & \text{if $m_i=n_i$} \\
                                                         |\mathcal{A}| & \text{otherwise.}
                                                       \end{array}
                                                  \right.
$$
\end{claim}

\begin{proof}
We only prove the case $i=1$ and the other cases can be shown similarly.
If $m_1=n_1$ and $m_1+h(\mathcal{A}, \hat{V_1})=|\mathcal{A}|$, then by Lemma~\ref{LEM: 2cycles} (i) and (ii), $\mathcal{A}$ can be partitioned into $\mathcal{A}_1\cup \mathcal{A}_2$ such that $h(\mathcal{A}_1,V_1)=n_1$ and $h(\mathcal{A}_1,\hat{V_1})=|\mathcal{A}|-n_1$, furthermore, for any $I\in \mathcal{A}_1,$ we have $|I\cap\hat{V_1}|\leq |\mathcal{A}|-n_1$. Thus, for every $I\in\mathcal{A}_1$, $|I\cap V_k|=|I\cap(V_1\cup V_k)|-|I\cap V_1|\ge n-(|\mathcal{A}|-n_1)-n_1=n_k$, i.e.  $\mathcal{A}_1\subset \mathcal{F}_k$. Hence $| \mathcal{F}_k|\geq 2n_k-1$.
By Proposition~\ref{PROP: p1}, we can find a rainbow independent  $n_k$-set $R_k'$ of $(\mathcal{F}_k, C_k)$. Let $\mathcal{F}'=\mathcal{F}\setminus C_{R_k'}$. Then $|\mathcal{F}'|=n-1-n_k$ and $|I\cap V(G-V_k)|\geq n-n_k$ for every $I\in\mathcal{F}'$. By the minimality of $G$, there is a rainbow independent $(n-1-n_k)$-set $R'$ of $(\mathcal{F}', G-V_k)$. Therefore, $R_k'\cup R'$ is a rainbow independent $(n-1)$-set  of $(\mathcal{F}, G)$, a contradiction.

Now assume $m_1+h(\mathcal{A}, \hat{V_1})\geq|\mathcal{A}|+2$ if $m_1=n_1$, or  $m_1+h(\mathcal{A}, \hat{V_1})\geq|\mathcal{A}|+1$ if $m_1<n_1$. By applying Lemma~\ref{LEM: 2cycles} (i) and (iii) to $V_1\cup\hat{V_1}$, we can partition $\mathcal{A}$ into $\mathcal{A}_1\cup \mathcal{A}_2$ such that (1) $h(\mathcal{A}_1,V_1)=n_1-1$, $h(\mathcal{A}_2,\hat{V_1})=|\mathcal{A}|+1-n_1$ and for any $I\in \mathcal{A}_2$, $|I\cap\hat{V_1}|\geq |\mathcal{A}|+2-n_1$  if $m_1=n_1$ (choose $\ell=1$) (which implies that $|\mathcal{A}_1|=n_1-1$), or $h(\mathcal{A}_1,V_1)=m_1$, $h(\mathcal{A}_2,\hat{V_1})=|\mathcal{A}|-m_1$ and for any $I\in \mathcal{A}_2$, $|I\cap\hat{V_1}|\geq |\mathcal{A}|+1-m_1$  if $m_1<n_1$ (choose $\ell=0$) (which implies that $|\mathcal{A}_1|=m_1$). By the assumption $f_{C_s}(m,m)=m$ for $s\ge 2m+1$, we have a rainbow independent $|\mathcal{A}_1|$-set $R_1$ of $(\mathcal{A}_1, C_1)$. On the other hand, by the minimality of $G$, we have a rainbow independent $|\mathcal{A}_2|$-set $R_2$ of $(\mathcal{A}_2, G-V_k-V_1)$. Hence $R_1\cup R_2\cup R_k$ is a rainbow independent $(n-1)$-set of $(\mathcal{F}, G)$, a contradiction too.

\end{proof}

If $k=3$, then  $m_i+h(\mathcal{A}, \hat{V_{i}})=m_i+h(\mathcal{A}, V_{3-i})=m_i+m_{3-i}$ for $i=1,2$. So we have either  $m_i=n_i$ for $i=1,2$ or $m_i<n_i$ for $i=1,2$ by Claim~\ref{2part}. For the former case, we have $n_1+n_2=m_1+m_2=|\mathcal{A}|+1=n-n_3+1$. So $t=n_1+n_2+n_3-n=1$, which is a contradiction to $t\ge k-1$. For the latter case,  we have $m_1+m_2=|\mathcal{A}|=n-n_3$. Applying Lemma~\ref{LEM: 2cycles} (i) and (ii) to $V_1\cup V_2$, we can partition $\mathcal{A}$ into $\mathcal{A}_1\cup \mathcal{A}_2$ such that $h(\mathcal{A}_1,V_1)=m_1<n_1$ and $h(\mathcal{A}_2, V_2)=|\mathcal{A}|-m_1=m_2<n_2$. By the assumption $f_{C_s}(m,m)=m$ for $s\ge 2m+1$, we obtain  a rainbow  independent $m_i$-set $R_i$ of $(\mathcal{A}_i, C_i)$ for $i=1,2$. Thus $R_1\cup R_2\cup R_3$ is a rainbow independent $(n-1)$-set of $(\mathcal{F}, G)$, a contradiction.

Now, we assume $k=4$. Hence $t\geq k-1=3$. We distinguish two cases according to the relations of $m_i$ and $n_i$ for $i\in[3]$.

If there exists some $i\in[3]$, say $i=3$, such that $m_3<n_3$. By Claim~\ref{2part}, $|\mathcal{A}|=m_3+h(\mathcal{A},V_1\cup V_2)$. First, we claim that we can choose $\mathcal{A}$ (recall that $\mathcal{A}=\mathcal{F}\setminus C_{R_4}$) with $m_3=h(\mathcal{A}, V_3)=n_3-1$. To show this, we choose $C_{R_4}$ such that $|\mathcal{F}_3 \cap\mathcal{A}|$ is as large as possible.
If  $|\mathcal{F}_3 \cap \mathcal{A}|\ge n_3-1$, then $h(\mathcal{A}, V_3)\ge n_3-1$, we are done. So, assume $|\mathcal{F}_3 \cap \mathcal{A}|<n_3-1$.
By applying Lemma~\ref{LEM: 2cycles} (i) and (ii) on $V_3\cup\hat{V_3}$, we can partition $\mathcal{A}$ into $\mathcal{A}_3\cup \mathcal{A}_3'$ such that $h(\mathcal{A}_3,V_3)=m_3$ and $h(\mathcal{A}_3',V_1\cup V_2)=n-n_4-m_3$, moreover, we can choose an $I_0\in \mathcal{A}_3'$ such that $|I_0\cap (V_1\cup V_2)|=n-n_4-m_3$ and $|I_0\cap V_3|\leq m_3$. This implies that $I_0\in\mathcal{F}_4$ but $I_0\notin\mathcal{F}_3$.
Therefore, reset $\mathcal{A}$ by replacing $I_0$ with some set of $\mathcal{F}_3\cap C_{R_4}$, we obtain a new $\mathcal{A}$ with  larger $|\mathcal{F}_3 \cap\mathcal{A}|$, a contradiction.
Now we have $m_3=h(\mathcal{A}, V_3)=n_3-1$ and so $h(\mathcal{A},V_1\cup V_2)=|\mathcal{A}|-m_3=n_1+n_2-t+1$. Again applying Lemma~\ref{LEM: 2cycles} (i) and (ii) to $V_3\cup\hat{V_3}$, we can partition $\mathcal{A}$ into $\mathcal{A}_3$ and  $\mathcal{A}_3'$ with $h(\mathcal{A}_3, V_3)=n_3-1$ and  $h(\mathcal{A}_3',V_1\cup V_2)=n_1+n_2-t+1$, furthermore, for any $I\in\mathcal{A}_3'$, $|I\cap V_3|\le n_3-1$ (this also implies that $|\mathcal{A}_3|=n_3-1$ and $|\mathcal{A}_3'|=n_1+n_2-t+1$).
By the assumption $f_{C_s}(m,m)=m$ for $s\ge 2m+1$, we can find a rainbow independent $(n_3-1)$-sets $R_3$ of $(\mathcal{A}_3, C_3)$.
We claim that there is at least one of $h(\mathcal{A}_3', V_i),\  i=1,2$ with $h(\mathcal{A}_3', V_i)=n_i$. Otherwise, we have $m_i'=h(\mathcal{A}_3', V_i)<n_i,\ i=1,2$. Applying  Lemma~\ref{LEM: 2cycles} (i) to $V_1\cup V_2$, we have $m_1'+m_2'\ge |\mathcal{A}_3'|$
and  $\mathcal{A}_3'$ can be partitioned into $\mathcal{B}_1$ and $\mathcal{B}_2$ with $h(\mathcal{B}_1,V_1)=m_1'$ and $h(\mathcal{B}_2, V_2)=|\mathcal{A}_3'|-m_1'\le m_2'$ (choose $\ell=0$).
By the assumption $f_{C_s}(m,m)=m$ for $s\ge 2m+1$, we have a rainbow independent $m_1'$-set $R_1$ of $(\mathcal{B}_1, C_1)$ and a rainbow independent $(|\mathcal{A}_3'|-m_1')$-set of $(\mathcal{B}_2, C_2)$.
Note that $$\sum_{i=1}^4|R_i|=m_1'+|\mathcal{A}_3'|-m_1'+n_3-1+n_4-1=n-1.$$
Hence $R_1\cup R_2\cup R_3\cup R_4$ is a rainbow independent $(n-1)$-set of $(\mathcal{F}, G)$, a contradiction.
By this claim, without loss of generality, assume $h(\mathcal{A}_3', V_1)=n_1$ and so $h(\mathcal{A}_3',V_2)=n_2-t+1$.
Again by Lemma~\ref{LEM: 2cycles} (i) and (ii), $\mathcal{A}_3'$ can be partitioned into $\mathcal{B}_1$ and $\mathcal{B}_2$ with $h(\mathcal{B}_1, V_1)=n_1$ and  $h(\mathcal{B}_2,V_2)=n_2-t+1$, furthermore, $|I\cap V_1|= n_1$ and $|I\cap V_2|\leq n_2-t+1$ for any $I\in \mathcal{B}_1$.
Hence $|I\cap V_4|=n-\sum_{i=1}^3|I\cap V_i|\ge n-n_1-(n_2-t+1)-(n_3-1)=n_4$ for any $I\in\mathcal{B}_1$ and so $\mathcal{B}_1\subset \mathcal{F}_1\cap \mathcal{F}_4$. Therefore, we have $|\mathcal{F}_4|\geq n_1+n_4> 2n_4-1$. By Proposition~\ref{PROP: p1}, we can find a rainbow independent  $n_4$-set $R_4'$ of $(\mathcal{F}_4, C_4)$. Let $\mathcal{F}'=\mathcal{F}\setminus C_{R_4'}$. Then $|\mathcal{F}'|=n-1-n_4$ and $|I\cap (V_1\cup V_2\cup V_3)|= n-n_4$ for every $I\in\mathcal{F}'$. By the minimality of $G$, there is a rainbow independent $(n-1-n_4)$-set $R'$ of $(\mathcal{F}', C_1\cup C_2\cup C_3)$. Therefore, $R_4'\cup R'$ is a rainbow independent $(n-1)$-set  of $(\mathcal{F}, G)$, a contradiction too.

Now assume  $m_i=n_i$ for all $i\in[3]$.
By Corollary~\ref{no0} (A), there is a rainbow independent $(n_3-1)$-set $R_3$ of $(\mathcal{F}_3\cap \mathcal{A}, C_3)$.
 Let $\mathcal{B}=\mathcal{A}\setminus C_{R_3}$, we choose $C_{R_3}$ minimize $\max\{n_1-|\mathcal{B}\cap\mathcal{F}_1|,n_2-|\mathcal{B}\cap\mathcal{F}_2|\}$.
By Claim \ref{2part}, $m_3+h(\mathcal{A},V_1\cup V_2)=|\mathcal{A}|+1$, i.e., $h(\mathcal{A}, V_1\cup V_2)=|\mathcal{B}|=n_1+n_2-t+1$.
As discussed in the above case, there is one of $h(\mathcal{B},V_i), \ i=1,2$ with $h(\mathcal{B},V_i)=n_i$.
Without loss of generality, assume $h(\mathcal{B},V_1)=n_1$ and so $h(\mathcal{B},V_{2})=n_{2}-t+1$.  Hence, $n_1-|\mathcal{B}\cap\mathcal{F}_1|\leq 0$ and $n_2-|\mathcal{B}\cap\mathcal{F}_2|\geq t-1\geq 2$. By Lemma~\ref{LEM: 2cycles} (i) and (ii), $\mathcal{B}$ can be  partitioned into $\mathcal{B}_1$ and $\mathcal{B}_2$ with $h(\mathcal{B}_1, V_1)=n_1$ and $h(\mathcal{B}_2, V_2)=n_2-t+1$, furthermore, for any $I\in\mathcal{B}_1$, we have $|I\cap V_1|=n_1$ and  $|I\cap V_2|\le n_2-t+1$.
Hence, for any $I\in\mathcal{B}_1$, we have $I\in \mathcal{F}_3$ or $I\in \mathcal{F}_4$ but $I\notin\mathcal{F}_2$. If there is some $I\in \mathcal{F}_3$, reset $\mathcal{B}$ by replacing $I$ with some set in $\mathcal{F}_2\cap C_{R_3}$ (which is can be done since $|\mathcal{F}_2\cap C_{R_3}|=|\mathcal{A}\cap\mathcal{F}_2|-|\mathcal{B}\cap\mathcal{F}_2|\ge n_2-|\mathcal{B}\cap\mathcal{F}_2|\ge 2$), the resulting new set $\mathcal{B}$ has smaller $\max\{n_1-|\mathcal{B}\cap\mathcal{F}_1|, n_2-|\mathcal{B}\cap\mathcal{F}_2|\}$, a contradiction. So for any $I\in \mathcal{B}_1$, we have $I\in\mathcal{F}_4$, i.e., $\mathcal{B}_1\subset \mathcal{F}_4$. Thus $|\mathcal{F}_4|\geq n_1+n_4-1\ge 2n_4-1$.
With the same discussion with the end of the above case, we have a contradiction to the assumption that $(\mathcal{F}, G)$ has no rainbow  independent $(n-1)$-set.

\end{proof}

\section{Proof of Theorem~\ref{large}}
In this section, $C_t$ always denotes a cycle with vertex set $[t]$ and edge set $\{12, 23, \ldots, (t-1)t, t1\}$.


\begin{proof}[Proof of Theorem~\ref{large}]
Let $\mathcal{F}\sqsubseteq\mathcal{I}_n(C_t)$ with $|\mathcal{F}|=n$. We show that there is a rainbow independent $n$-set of  $(\mathcal{F}, C_t)$ if $t>\frac{1}{3}n^2+\frac{44}{9}n$.
Suppose to the contrary that $(C_t, \mathcal{F})$ has no rainbow independent $n$-set. Choose a member $J\in\mathcal{F}$. By Corollary~\ref{no0} (A), $( \mathcal{F}\setminus\{J\}, C_t)$ has a rainbow independent $(n-1)$-set $R$. Let $R'=R\cap N_{C_t}[J]$. Then $R'$ is a rainbow independent set of $( \mathcal{F}\setminus\{J\}, C_t)$ too. Choose such a $R$ with the smallest $|R'|$. We claim that $|R'|\ge\lceil\frac{n}{2}\rceil$. Otherwise, suppose $|R'|<\frac{n}{2}$. Since  $|N_{C_t}[i]\cap N_{C_t}[j]|\le 1$ for any two vertices $i,j\in J$ with $i\not=j$, every vertex of $R'$ is contained in at most two members of $\{N_{C_t}[i] : i\in J\}$. So there exists a $j_0\in J$ with $N_{C_t}[j_0]=[j_0-1,j_0+1]\cap R=\emptyset$. So we can enlarge $R$ by adding $j_0$ with color $J$ to $R$, a contradiction to the assumption.
Now let $F=C_t-N_{C_t}[J\cup R]$. By the definition of $R$ and $R'$,  $J\subseteq N_{C_t}[R']\subseteq N_{C_t}[R]$. So we have
$$|N_{C_t}[J\cup R]|=|N_{C_t}[J]|+|N_{C_t}[R]|-|N_{C_t}[J]\cap N_{C_t}[R]|\le 3n+3n-|J|=5n.$$
Recall that $C_{R'}$ is the  set of the corresponding colors  assigned to vertices in $R'$. We claim that $C_{\mathcal{F}}(i)\cap C_{R'}=\emptyset$ for any $i\in V(F)$.
If not, assume there is an $I\in C_{\mathcal{F}}(i)\cap C_{R'}$ for some $i\in V(F)$ and assume $I$ is the color of $j$ in $R'$, i.e. $C_{R'}(j)=\{I\}$.
Let $\tilde{R}=(R\setminus\{j\})\cup \{i\}$. Then $\tilde{R}'=\tilde{R}\cap N_{C_t}[J]=R'\setminus\{j\}$,
a contradiction to the minimality of $|R'|$.
This claim implies that for any $I\in C_{R'}\cup\{J\}$, we have $I\subseteq N_{C_t}[J\cup R]$. Let $A=\{(i,I) :  i\in I \mbox{ and } I\in\mathcal{F}\}$, $B=\{(i,I) : i\in V(F) \mbox{ and } i\in I, I\in\mathcal{F}\}$ and $C=\{(i,I) : i\notin V(F) \mbox{ and } i\in I, I\in\mathcal{F}\}$. So, with a double-counting argument, we have
\begin{equation}\label{EQ:e1}
{|B|=\sum_{i\in V(F)}c_{\mathcal{F}}(i)=|A|-|C|\le n^2-n|C_{R'}\cup\{J\}|\le\frac{1}{2}n^2-n\mbox{,}}
\end{equation}
where the inequality holds since $|C_{R'}|=|R'|\ge \lceil\frac n2\rceil$.
Note that $N_{C_t}[J\cup R]$ is a union of intervals of length at least $3$ and $|N_{C_t}[J\cup R]|\le 5n$. So $N_{C_t}[J\cup R]$ consists of at most $\frac{5n}{3}$ intervals. This implies that $F=C_t-N_{C_t}[J\cup R]$ has at least $t-5n$ vertices and consists of $m\le\frac{5n}{3}$ paths, say $P_1,P_2,\ldots,P_m$. By Corollary~\ref{no0} (B1), $c_{\mathcal{F}}(a)\ge1$ for any $a\in[1,t]$. By Corollary~\ref{no0} (B3),
every path $P_j$ contains at most $\frac{|V(P_j)|+1}2$ vertices $a$ with $c_{\mathcal{F}}(a)=1$.
Hence,
\begin{equation}\label{EQ: e2}
\begin{split}
\sum_{i\in F}c_{\mathcal{F}}(i) &=\sum_{j=1}^m\sum_{i\in V(P_j)}c_{\mathcal{F}}(i) \\
 &\ge \sum_{j=1}^m\left(2|V(P_j)|-\frac{|V(P_j)|+1}{2}\right)\\
 &=\frac32\sum_{j=1}^m|V(P_j)|-\frac {m}{2}\\
 &\ge\frac32(t-5n)-\frac{m}{2}\\
&\ge\frac{3}{2}t-\frac{25}{3}n\mbox{.}
\end{split}
\end{equation}
From (\ref{EQ:e1}) and (\ref{EQ: e2}), we have
$$\frac{3}{2}t-\frac{25}{3}n\le\sum_{i\in F}c_{\mathcal{F}}(i)\le\frac{1}{2}n^2-n\mbox{,}$$
i.e. $t\le\frac{1}{3}n^2+\frac{44}{9}n$, a contradiction.
\end{proof}

\section{Proof of Theorem~\ref{+2}}
Let $\mathbb{Z}_t$ be the additive group of remainder of modulo $t$.
For  $a,b \in\mathbb{Z}_t$, $a>b$ means $a\pmod t> b\pmod t$.
For two sequences $(a_1,\ldots,a_n), (b_1,\ldots,b_n)\in \mathbb{Z}_t^n$, $(a_1,\ldots,a_n)>(b_1,\ldots, b_n)$ if and only if there exists some $i\in[n]$ such that $a_i>b_i$ and $a_j=b_j$ for all $j<i$.

Throughout this section, let  $C_t$ be  a cycle with vertex set $\mathbb{Z}_t$  and edge set $\{01,12,...,(t-1)0\}$. If $t=2n+1$, by Theorem~\ref{THM: Lu}, the result holds. So we assume $t\ge 2n+2$ in the following.
Let $\mathcal{F}=(B_1,\ldots,B_n)\sqsubseteq\mathcal{I}_n^2(C_t)$.
Choose a maximal rainbow independent set $A =\{a_1,a_2,\ldots,a_r\}$ of $(C_t, \mathcal{F})$ with the property that
$a_1,\ldots, a_r$ are in a clockwise order in $C_t$ and
$D_A=(a_{r}-a_1,a_2-a_1, a_3-a_2,\ldots,a_r-a_{r-1})$ is minimal.
Without loss of generality, assume $C_A=\{B_1, B_2, \ldots, B_r\}$ and $a_1 = 1$. Then we have $a_1<a_2<\ldots<a_r$ by the assumption of $A$. It is sufficient to show that $r=n$.
By Corollary~\ref{no0} (A), we know $r\ge n-1$.
Suppose $r=n-1$.
Define  $A_i=\{a_i, a_i+1\}$ for $i=1,2,\ldots, n-1$.
We have the following claims.

\begin{claim}\label{claim:1}
  $B_n \subseteq A\cup(A+1)\cup \{ a_1-1\}$. Moreover, $0\in B_n$ and $|B_n\cap A_i|=1$ for each $i\in[n-1]$.
\end{claim}
\begin{proof}
	If there exists one $b\in B_n$ such that $b\notin A\cup(A+1)\cup(A-1)$, then $ \{b\}\cup A$ is a larger rainbow independent set than $A$, a contradiction. Now suppose that there exists one $b\in B_n$ and an $i\in[2,r]$ with $b=a_i - 1$ but $b \ne a_{i-1}+1$. Then we can replace $a_i$ by $b$ in $A$ and get another  rainbow independent set $\{a_1, \ldots, a_{i-1},b,a_{i+1},\ldots, a_r\}$ of $(\mathcal{F}, C_t)$ with $b-a_{i-1}<a_i-a_{i-1}$, a contradiction to the minimality of $D_A$. So $B_n \subseteq A\cup(A+1)\cup \{a_1-1\}$.

Note that $A\cup(A+1)$ consists of exactly two independent sets of size $n-1$. Since $|B_n|=n$, we have $a_1-1\in B_n$, i.e. $0\in B_n$ and $|B_n\cap\{a_i, a_i+1\}|=1$ for any $i\in[n-1]$.
\end{proof}

\begin{claim}\label{claim:3}
If $C_A(a_i)=B_{k_i}$ and $B_n\in C_{\mathcal{F}}(a_i)$ for some $i\ge 2$, then $B_{k_i}=B_n$.
\end{claim}
\begin{proof}

If $B_{k_i}\not=B_n$, then we can reset the color of $a_i$ by $B_n$ and apply Claim~\ref{claim:1} to $B_{k_i}$, we have $0\in B_{k_i}$ and $|B_{k_i}\cap A_\ell|=1$ for every $\ell\in[n-1]$. Since $B_{k_i}\ne B_n$, we have $s(B_{k_i})\not=s(B_n)$. Without loss of generality, we assume $0\le s(B_{k_i})<s(B_n)$.
Let $B_n=(j, j+2,\ldots, 0, 2,\ldots,  2n-t+j-2)$. Then $j>0$ and $2n-t+j-2<j$ (otherwise, we have $t\le 2n-2$, a contradiction).
So, by Claim~\ref{claim:1}, $t-2$ must be in $A_{n-1}$ and $j$ must be in $A_s$, where $s=n-\frac{t-j}2$.
If $s(B_{k_i})=0$ then $B_{k_i}=(0, 2, \ldots, 2n-2)$.
So we have $2n-2\in A_{n-1}$, i.e. $a_{n-1}=2n-3$ or $2n-2$. In any case, we have $t\le 2n+1$, a contradiction.
Now assume $s(B_{k_i})=h\not=0$. Then $B_{k_i}=\{h, h+2, \ldots, 0, 2, \ldots, 2n-t+h-2\}$.
Hence $h$ and $j$ have the same parity. So $j-2\in B_{k_i}$ but $j-2\notin B_n$ because $j-(2n-t+j-2)=t-2n+2\ge 4$.
Since $j\in B_n\cap A_s$, we have $2n-t+j-2$ must be in $B_n\cap A_{s-1}$. Since $A_s\cap A_{s-1}=\emptyset$ and $2n-t+j-2<j-2<j$, there is no $A_{\ell}$ containing $j-2$ for $\ell\in [n-1]$. This is a contradiction  to $|B_{k_i}\cap A_{\ell}|=1$ for all $\ell\in [n-1]$.
\end{proof}
If $s(B_n)=0$, i.e., $B_n = \{0,2,\ldots,2n-2\}$.
Since $A$ is an independent set with $a_1=1\in A$ and $|B_n\cap \{a_i,a_{i}+1\}|= 1$ for every $i\in[n-1]$, there exists an integer $k$ with $1<k\le n$ such that $a_i = 2i-1$ for $1\le i< k$ and  $a_i=2i$ for $i\ge k$ (if $k<n$).
By Claim~\ref{claim:3}, we have $ B_i = B_n$ for all $i\ge k$. Clearly, $2n\notin B_i$ for all $i\in [k, n]$.
If $t=2n+2$, since $a_i=2i-1\in B_i$ and $B_i$ is 2-jump, we have  $2n\notin B_i$ for all $i\in [k-1]$. Therefore, $2n\notin B_i$ for all $i\in [n]$, which is a contradiction to Corollary~\ref{no0} (B1).
If $t> 2n+2$, then $t-3 > 2n-1>a_{n-1}$. Clearly, $t-3\notin B_n $.  Let $\ell \le k-1 $ be the minimal number with $t-3 \in B_\ell$ (If such $\ell$ does not exist, then $t-3\notin B_i$ for all $i\in [n]$, a contradiction to Corollary~\ref{no0} (B1), too). Since $B_i\in\mathcal{I}_n^2(C_t)$  with $a_i=2i-1\in B_i$ and $t-3\notin B_i$ for any $i< \ell$, we have $a_i +2 = a_{i+1} \in B_i$. So we can recolor $a_{i+1}$ with $B_i$ for all $i\in[\ell-1]$, remain the color of $a_i$ unchanged for $i\in[\ell+1, n-1]$ and color $t-3$ with $B_\ell$ and 0 with $B_n$, i.e. $(A\setminus \{1\})\cup\{0, t-3\} $ is a rainbow independent $n$-set of $(\mathcal{F}, C_t)$, a contradiction.

Now suppose $s(B_n)= -2(n-m)$ for some $n> m>0$, i.e., $B_n =\{t-2(n-m),\ldots, t-2, 0, 2, \ldots, 2m-2\}$.
If $t = 2n+2$, then $B_n=\{2m+2,\ldots, 2n, 0, 2, \ldots, 2m-2\}$. Clearly, $2m-1\notin A$ (otherwise, $B_n\cap\{2m-1, 2m\}=\emptyset$, a contradiction to Claim~\ref{claim:1}). With the same reason as the case $s(B_n)=0$, we have an integer $k$ with $1<k\le m$ such that $a_i = 2i-1$ for any $i\in[k-1]$ and  $a_i=2i$ for any $i\in [k, m-1]$.
With the same discussion with the case $s(B_n)=0$, we have $2m\notin B_i$ for all $i\in [m-1]$.
For $i=m$, we have $a_m\not=2m$ (otherwise, $B_n\cap\{a_m, a_m+1\}=\emptyset$, a contradiction to Claim~\ref{claim:1}).
So $a_i>2m$ for any $i\in[m, n-1]$.
Since $B_j$ is 2-jump, all elements of $B_j$ have the same parity. By Claim~\ref{claim:3}, $B_j=B_n$ if $a_j\in B_j$ is an even numbers when $j\in [m,n-1]$.
So $2m\notin B_j$ for any $j\in[m,n-1]$.
Therefore, we have $2m\notin B_i$ for all  $i\in [n]$, a contradiction to Corollary~\ref{no0} (B1).
If $t> 2n+2$, then by Claim~\ref{claim:1}, we have $a_{i}\in \{t-2n+2i , t-2n+2i-1\}$ for $i\in[m, n-1]$, and $a_i \in\{ 2i-1, 2i\}$ for $i\in[2,m-1]$. In particular we have $a_{n-1} \in\{-2,-3\}$. If $a_{n-1}=-2$, then $a_1-a_{n-1}=3$ and  $a_{m}-a_{m-1}\ge t-2n+2m-1- 2(m-1)>3$, we have a contradiction to the minimality of $D_A$ by reordering $A$ with $A = \{a_{m},a_{m+1} ...,a_{n-1}, a_1,...a_{m-1}\}$. So we have $a_{n-1} = -3$, which implies that  $a_{i}=t-2n+2i-1$ for every $i\in[m, n-1]$. In particular, $a_m=t-2(n-m)-1\in B_m$. Note that $a_m-2=t-2(n-m)-3> 2m-1>a_{m-1}$. By the minimality of $D_A$, we have $-2(n-m)-3 \notin B_m$ (otherwise, we can reduce $a_m-a_{m-1}$ by replacing $a_m$ with $a_m-2$ in $A$). Since $B_m$ is 2-jump, we have  $s(B_m)=t-2(n-m)-1$ and so $-1 \in B_m$.
Then $A'= (A\setminus \{ a_m\})\cup \{ -1 \} = \{a_{m+1}, \ldots, a_{n-1}, -1, a_1, \ldots, a_{m-1}\}$ is a rainbow independent set of $(\mathcal{F}, C_t)$ with $a_{m+1}-a_{m-1}\ge t-2n+2m+1-2(m-1)>5>4=a_{1}-a_{n-1}$, a contradiction to the minimality of $D_A$.

This completes the proof of  Theorem~\ref{+2}.
\section{Concluding remarks and discussions}
In this article, we show that (1) Conjecture~\ref{CONJ: 1} is true when $|V(G)|\in\{ 2n-1, 2n\}$ and (2) if Conjecture~\ref{conj} is ture then Conjecture~\ref{CONJ: 1} holds for graphs $G\in\mathcal{D}(2)$ with $c_e(G)\le 3$. We believe that Conjecture~\ref{conj} implies Conjecture~\ref{CONJ: 1}, we leave this as an open problem.
For Conjecture~\ref{conj}, it has been shown that this conjecture is true for the base case (Theorem~\ref{THM: Lu}) and for large $t$ (Theorem~\ref{large}).
Could we show the rest case of Conjecture~\ref{conj}?

\end{document}